\documentclass[12pt]{article}
\usepackage{amsmath,amscd,amsthm,amsfonts,amssymb}
\begin{document}
\newtheorem{thm}{Theorem}
\newtheorem{cor}{Corollary}
\newtheorem{lem}{Lemma}
\newtheorem{slem}{Sublemma}
\newtheorem{prop}{Proposition}
\newtheorem{defn}{Definition}
\newtheorem{conj}{Conjecture}
\newtheorem{ques}{Question}
\newtheorem{claim}{Claim}
\newtheorem{rmk}{Remark}
\newcounter{constant} 
\title{Desingularization of branch points of minimal surfaces in $\mathbb{R}^4$(II)}
\author{Marina Ville}
\date{ }
\newcommand{\newconstant}[1]{\refstepcounter{constant}\label{#1}}
\newcommand{\useconstant}[1]{C_{\ref{#1}}}
    \renewcommand{\arraystretch}{2}
    \maketitle
\begin{abstract}
We desingularize a branch point $p$ of a minimal disk $F_0(\mathbb{D})$ in $\mathbb{R}^4$ through immersions $F_t$'s which have only transverse double points and are branched covers of the plane tangent to $F_0(\mathbb{D})$ at $p$. If $F_0$ is a topological embedding and thus defines a knot in a sphere/cylinder around the branch point,  the data of the double points of the $F_t$'s give us a braid representation of this knot as a product of bands.
\end{abstract}
\section{Introduction} 
\subsection{The purpose}
Minimal surfaces in $\mathbb{R}^4$ are immersed except at branch points, near which the surface is a $N$-branched covering of the  tangent plane at the branch point (for some $N>1$). In [Vi 2] we looked at a minimal map $F_0:\mathbb{D}\longrightarrow\mathbb{R}^4$ with a branch point at the origin and we described how to desingularize $F_0$          
 through minimal immersions $F_t$'s with only transverse double points. However, unlike $F_0$ these $F_t$'s were not branched coverings of the disk. We discuss here a desingularization through immersions which are not necessary minimal but which remain $N$-branched coverings of the disk.\\
 If $F_0$ is a topological embedding, we recall that the intersection of $F_0(\mathbb{D})$ with a small sphere (equivalently a small cylinder) centered at the branch point defines a knot which has a representation as a $N$-braid (cf. [S-V] mimicking the construction of [Mi]). In that case, we will use a construction of Rudolph to show how the double points of the immersions $F_t$'s appear in a {\it band representation} of this braid (i.e.
an expression of the braid as a product  of conjugates of braid generators and of their inverses).
\subsection{The setting}
We consider  a branched immersion
$$F_0:\mathbb{D}\longrightarrow \mathbb{R}^4\cong \mathbb{C}^2\times\mathbb{C}^2$$
\begin{equation}\label{branched immersion}
F_0:z\mapsto (z^N+h_1(z),h_2(z))
\end{equation}
where, for $i=1,2$, $h_i:\mathbb{D}\longrightarrow \mathbb{C}$ is a function with $|h_i(z)|=o(|z|^N)$.\\
It is standard (cf. for example [G-O-R]) to introduce a function $w:\mathbb{D}\longrightarrow\mathbb{D}$ such that
\begin{equation}\label{definition de w}
w(z)^N=z^N+h_1(z)
\end{equation}
and which verifies
\begin{equation}\label{propriete de w}
w=z+o(|z|)\ \ \ \ \ \ \ \ z=w+o(|w|)
\end{equation}
Possibly after restricting ourselves to a smaller disk centered at $0$, we reparametrize $\mathbb{D}$ with $w$ so we can
 rewrite $F$ in terms of $w$ as
\begin{equation}\label{branched immersion}
F_0:w\mapsto (w^N,h(w))
\end{equation}
where $h(w)=o(|w|^N)$.
\begin{rmk}
Throughout this paper, we only use the fact that $F_0$ is a real analytic branched immersion, not that it is minimal. We could probably also do without the real analytic assumption.
\end{rmk} 
\subsection{The construction}
For $\lambda,\mu$ small complex numbers, we will be considering the immersions
\begin{equation}\label{definition de F lambda}
F_{\lambda,\mu}:w\mapsto (w^N,h(w)+\lambda w+\mu\bar{w})
\end{equation}
possibly adding a small correction term if necessary:
\begin{equation}\label{F lambda mu gamma}
F_{\lambda,\mu,\gamma}:w\mapsto \big(w^N,h(w)+\lambda w+\mu\bar{w}+Re(\gamma w^2)\big)
\end{equation}
where $\gamma$ is very small compared to $\lambda$ and $\mu$ and is only introduced to give more wiggle room for transversality arguments.
\begin{rmk} We used immersions similar to (\ref{definition de F lambda}) in [Vi 1] where we established a
connection between the algebraic crossing number of the braid and the normal bundle of the branched disk in an ambient $4$-manifold.
\end{rmk}
The paper is devoted to proving the following:
\begin{thm}\label{desingularisation par bands} For $\lambda, \mu$ generic and small enough, $F_{\lambda,\mu}$ has a finite number of crossing points $m_1,...,m_n$, all transverse.\\
Assume that $F_0$ is a topological embedding and let $K$ be the knot defined by the branch point.
If $\frac{\mu}{\lambda}$ is small enough (resp. large enough), the knot $K$ is represented by a $N$-braid $\beta$ which is the product of the following pieces:
\begin{enumerate}
\item
$$\prod_{2k,\ 2\leq 2k\leq N-1}\sigma_{2k}$$ $$\mbox{(resp.}\ \ \ \  \prod_{2k,\ 2\leq 2k\leq N-1}\sigma^{-1}_{2k})$$

\item
$$\prod_{2k+1,\ 1\leq 2k+1\leq N-1}\sigma_{2k+1}$$ $$\mbox{(resp.}\ \ \ \  \prod_{2k+1,\ 1\leq 2k+1\leq N-1}\sigma^{-1}_{2k+1})$$

\item
for every double point $m_1,...,m_n$ of $F_{\lambda,\mu}$, one copy of
\begin{equation}
b(m_i)\sigma_{k(m_i)}^{2\epsilon(m_i)}b(m_i)^{-1}
\end{equation}
where 
\begin{itemize}
\item
$\epsilon(m_i)$ is the sign of the intersection point $m_i$
\item
$k(m_i)\in\{1,...,N-1\}$
\item
$b(m_i)$ is some element of the braid group $B_N$.
\end{itemize}
\end{enumerate}
\end{thm}

\subsection{Trivial knots}
It follows from the expression of the braid that, if $F_{\lambda,\mu}$ is an embedding for $\frac{\lambda}{\mu}$ large enough or small enough, the knot $K$ is trivial. 
There exist branched minimal disks with corresponding knots which are non trivial but have  $4$-genus $0$, for example $10_{155}$ (cf. [S-V]). For such a knot, the signed number of double points of the $F_{\lambda,\mu}$ (for $\frac{\lambda}{\mu}$ large enough or small enough) is zero but $F_{\lambda,\mu}$ has necessarily double points. 
\subsection{Sketch of the paper}
We will first establish some properties of the $F_{\lambda,\mu}$'s, for  generic $\lambda$, $\mu$'s; then we will construct a closed loop $\Gamma$ in the plane $\Pi_2$ generated by  the first two coordinates. The braid $\beta$ considered in Th. \ref{desingularisation par bands} will be defined as 
\begin{equation}\label{definition de la tresse}
\beta=\pi_2^{-1}(\Gamma)\cap F_{\lambda,\mu}(\mathbb{D})
\end{equation}
for $\lambda,\mu$ small enough and where 
\begin{equation}\label{definition de pi2}
\pi_2:\mathbb{R}^4\longrightarrow\Pi_2
\end{equation}
is the orthogonal projection.
\section{The family of immersions}\label{The family of immersions}
\begin{lem}\label{lambda generique}
For generic $\lambda$, $\mu$'s, the following is true:\\
if $w_1, w_2\in\mathbb{D}$ verify $w_1\neq w_2$ and $F_{\lambda,\mu}(w_1)=F_{\lambda,\mu}(w_2)$, then the two tangent planes 
$dF_{\lambda,\mu}(w_1)(\mathbb{R}^2)$ and $dF_{\lambda,\mu}(w_2)(\mathbb{R}^2)$ are transverse.
\end{lem}
\begin{rmk}\label{trois points} Lemma \ref{lambda generique} does not exclude the possibility of triple points (i.e. three disks meeting at a point, every two of them transversally): we will see that later. 
\end{rmk}
\begin{proof}
The proof is based on the Transversality Lemma. We introduce
$$\Phi:\mathbb{C}\times\mathbb{C}\times \mathbb{D}\times \mathbb{D}\longrightarrow \mathbb{R}^4\oplus \mathbb{R}^4$$
\begin{equation}\label{definition de Phi pour les bandes}
\Phi(\lambda,\mu,w_1,w_2)=(F_{\lambda,\mu} (w_1), F_{\lambda,\mu}(w_2))
\end{equation}
and we check that it is transverse to the diagonal $\Delta_8$ of $\mathbb{R}^4\oplus \mathbb{R}^4$ for $w_1\neq w_2$. We derive from the basis $(e_1,e_2,e_3,e_4)$ of $\mathbb{R}^4$ (in which (\ref{branched immersion}) is written) a basis 
$$(e_1^{(1)},...,e_4^{(1)},e_1^{(2)},...,e_4^{(2)})$$ of
$\mathbb{R}^4\oplus \mathbb{R}^4$; thus the diagonal $\Delta_8$ is generated by 
$(e_1^{(1)}+e_1^{(2)},...,e_4^{(1)}+e_4^{(2)})$. \\
A point in the preimage of $\Delta_8$ {\it via} $\Phi$ is of the form $(\lambda,\mu,w_1,w_2)$, where $w_2=\nu w_1$ for a complex number $\nu$ verifying 
$$\nu^N=1.$$ We introduce real coordinates by setting
\begin{equation}\label{coordinates}
\lambda=\lambda_1+\lambda_2\ \ \ \ \ \ w_1=x_1+iy_1\ \ \ \ \ \ w_2=x_2+iy_2
\end{equation}
and we compute the following determinant at points $w_1, w_2=\nu w_1$; the subscripts denote the components in the  basis $(e_1,e_2,e_3,e_4)$:
$$det(\frac{\partial\Phi}{\partial x_1}, \frac{\partial\Phi}{\partial x_2}, \frac{\partial\Phi}{\partial \lambda_1}, \frac{\partial\Phi}{\partial \lambda_2},e_1^{(1)}+e_1^{(2)},e_2^{(1)}+e_2^{(2)},e_3^{(1)}+e_3^{(2)},e_4^{(1)}+e_4^{(2)})=$$
\[\left |\begin{array}{cccccccc}
(\frac{\partial F}{\partial x_1})_1&0&0&0&1&0&0&0\\
(\frac{\partial F}{\partial x_1})_2&0&0&0&0&1&0&0\\
(\frac{\partial F}{\partial x_1})_3&0&x_1&-y_1&0&0&1&0\\
(\frac{\partial F}{\partial x_1})_4&0&y_1&x_1&0&0&0&1\\
0&(\frac{\partial F}{\partial x_2})_1&0&0&1&0&0&0\\
0&(\frac{\partial F}{\partial x_2})_2&0&0&0&1&0&0\\
0&(\frac{\partial F}{\partial x_2})_3&x_2&-y_2&0&0&1&0\\
0&(\frac{\partial F}{\partial x_2})_4&y_2&x_2&0&0&0&1\\
\end{array}
\right |=
\left |\begin{array}{cccc}
(\frac{\partial F}{\partial x_1})_1&-(\frac{\partial F}{\partial x_2})_1&0&0\\
\\
(\frac{\partial F}{\partial x_1})_2&-(\frac{\partial F}{\partial x_2})_2&0&0\\
\\
(\frac{\partial F}{\partial x_1})_3&-(\frac{\partial F}{\partial x_2})_3&x_1-x_2&y_2-y_1\\
\\
(\frac{\partial F}{\partial x_1})_4&-(\frac{\partial F}{\partial x_2})_4&y_1-y_2&x_1-x_2\\
\end{array}
\right |\]
$$=[(x_1-x_2)^2+(y_1-y_2)^2][-(\frac{\partial F}{\partial x_1})_1(\frac{\partial F}{\partial x_2})_2+
(\frac{\partial F}{\partial x_2})_1(\frac{\partial F}{\partial x_1})_2]$$
\begin{equation}\label{formule finale avec l'imaginaire pour le gros determinant}
=|1-\nu|^2N^2|w|^{2N}Im(\nu).
\end{equation}
Similarly we show
$$det(\frac{\partial\Phi}{\partial x_1}, \frac{\partial\Phi}{\partial y_2}, \frac{\partial\Phi}{\partial \lambda_1}, 
\frac{\partial\Phi}{\partial \lambda_2},e_1^{(1)}+e_1^{(2)},e_2^{(1)}+e_2^{(2)},e_3^{(1)}+e_3^{(2)},e_4^{(1)}+e_4^{(2)})=$$
\begin{equation}\label{formule finale avec la partie reelle pour le gros determinant}
=-|1-\nu|^2N^2|w|^{2N}Re(\nu)
\end{equation}
Lemma \ref{lambda generique} follows from (\ref{formule finale avec l'imaginaire pour le gros determinant}) and (\ref{formule finale avec la partie reelle pour le gros determinant}).
\end{proof}
We denote by $\Pi_3$ the $3$-plane $\Pi_3$ generated by the first $3$ coordinates; we let
\begin{equation}
\pi_3:\mathbb{R}^4\longrightarrow \Pi_3
\end{equation}
be the orthogonal projection and we show a lemma similar to Lemma \ref{lambda generique} for $\pi_3\circ F_{\lambda,\mu}$:
\begin{lem}\label{points fixes ds les 3 premieres coordonnes}
For generic $\lambda,\mu$, the following is true:\\
if $w_1,w_2\in\mathbb{D}$ verify $w_1\neq w_2$ and $(\pi_3\circ F_{\lambda,\mu})(w_1)=(\pi_3\circ F_{\lambda,\mu})(w_2)$, then the two tangent planes $\pi_3\big(dF_{\lambda,\mu}(w_1)(\mathbb{R}^2)\big)$ and $\pi_3\big(dF_{\lambda,\mu}(w_2)(\mathbb{R}^2)\big)$ are transverse.
\end{lem}
\begin{proof}
Similarly to above, we introduce the map
$$\Psi:\mathbb{C}\times \mathbb{C}\times\mathbb{D}\times\mathbb{D}\longrightarrow \mathbb{R}^3\oplus\mathbb{R}^3$$
\begin{equation}
\Psi:(\lambda,\mu,w_1,w_2)\mapsto\Big((\pi_3\circ F_{\lambda,\mu})(w_1), (\pi_3\circ F_{\lambda,\mu})(w_2)\Big)
\end{equation}
By truncating the determinants appearing in the proof of Lemma \ref{lambda generique}, we get
$$det(\frac{\partial\Psi}{\partial x_1}, \frac{\partial\Psi}{\partial x_2}, \frac{\partial\Psi}{\partial \lambda_1}, 
e_1^{(1)}+e_1^{(2)},e_2^{(1)}+e_2^{(2)},e_3^{(1)}+e_3^{(2)})=(x_1-x_2)N^2|w|^{2N-2}$$
$$det(\frac{\partial\Psi}{\partial x_1}, \frac{\partial\Psi}{\partial x_2}, \frac{\partial\Psi}{\partial \lambda_2}, 
e_1^{(1)}+e_1^{(2)},e_2^{(1)}+e_2^{(2)},e_3^{(1)}+e_3^{(2)})=-(y_1-y_2)N^2|w|^{2N-2}$$
Hence $\Psi$ is transverse to the diagonal $\Delta_6$ of $\mathbb{R}^3\oplus\mathbb{R}^3$ and Lemma \ref{points fixes ds les 3 premieres coordonnes} follows.
\end{proof}
Next we show that $\pi_3\circ F_{\lambda,\mu}$ has only a finite number of triple points:
\begin{lem}\label{N moins un preimages}
Let 
\begin{equation}
\nu=e^{\frac{2\pi }{N}i}
\end{equation}
 and let $k,l$ be two different integers in $\{1,...,N-1\}$. For generic $\lambda, \mu$, 
there is a finite number of points $w\in\mathbb{D}$ such that
\begin{equation}\label{triple points}
Re\big(\lambda w+\mu \bar{w}+h(w)\big)=Re\big(\lambda \nu^k w+\mu\bar{\nu}^k\bar{w}+h(\nu^k w)\big)=Re\big(\lambda \nu^l w+\mu\bar{\nu}^l\bar{w}+h(\nu^l w)\big)
\end{equation}
\end{lem}
\begin{proof}
We let
$$\psi:\mathbb{C}\times\mathbb{D}\longrightarrow\mathbb{R}^2$$
$$\psi(\lambda,w)=
\Big(Re\big[\lambda (1-\nu^k) w+\mu (1-\bar{\nu}^k)\bar{w}+h(w)-h(\nu^k w)],$$
$$Re[\lambda (1-\nu^l) w+\mu (1-\bar{\nu}^l)\bar{w}+h(w)-h(\nu^l w)]\Big).$$
We define
\begin{equation}
(1-\nu^k)w=w_1^{(k)}+iw_2^{(k)}\ \ \ \ (1-\nu^l)w=w_1^{(l)}+iw_2^{(l)}
\end{equation}
and compute

\[ det(\frac{\partial \psi}{\partial \lambda_1},\frac{\partial \psi}{\partial \lambda_2})=
\left| \begin{array}{cc}
w_1^{(k)}&-w_2^{(k)}\\
w_1^{(l)}&-w_2^{(l)}
\end{array}\right|=Im[(1-\nu^k)w(1-\bar{\nu}^l)\bar{w}]
\]

$$
=|w|^2\Big[\sin ({\frac{2\pi}{N}l})- \sin ({\frac{2\pi}{N}k}) +\sin \big({2\frac{\pi}{N}(k-l)}\big)\Big]
$$
\begin{equation}\label{determinant ds R2}
=4|w|^2\sin ({\frac{\pi}{N}l}) \sin ({\frac{\pi}{N}k}) \sin \big({\frac{\pi}{N}(k-l)}\big)
\end{equation}
which is not zero. We use the Transversality Lemma again and conclude that for a generic $\lambda$, $\psi(\lambda,.)$ is transverse to $(0,0)$, that is, $(0,0)$ is attained at a finite number of points. 
\end{proof}
NOTATIONS. We remind the reader that $\pi_2$ is the projection onto the plane $\Pi_2$ generated by the first two coordinates. \\
We denote by $W_{\lambda,\mu}$ the set of $w$'s in $\mathbb{D}\backslash\{0\}$ which verify (\ref{triple points}) for some $k,l$ and we let
\begin{equation}\label{definition de X lambda mu}
X_{\lambda,\mu}=(\pi_2\circ F_{\lambda,\mu})(W_{\lambda,\mu})
\end{equation} 
If $D_{\lambda,\mu}\in\mathbb{R}^4$ is the set of double points of $F_{\lambda,\mu}$, we let 
\begin{equation}\label{definition de la projection des points doubles}
{\mathcal D}_{\lambda,\mu}=\pi_2(D_{\lambda,\mu})
\end{equation}
In the following lemma, we use $F_{\lambda,\mu,\gamma}$ defined in (\ref{F lambda mu gamma}); nevertheless we keep the notation 
${\mathcal D}_{\lambda,\mu}$ and $X_{\lambda,\mu}$ in order not to burden the notations.
\begin{lem}\label{pas d'intersection des ensembles finis}
For generic $\lambda$'s, $\mu$'s, $\gamma$ 
$${\mathcal D}_{\lambda,\mu}\cap X_{\lambda,\mu}=\varnothing$$
In particular $F_{\lambda,\mu}$ does not have triple points (cf. Remark \ref{trois points}).
\end{lem}
\begin{proof}
We pick a very small positive number $\eta$ (how small it need to be will be clear from the proof below) and
given a $7$-uple 
$A=(a,b,c,d,e,\alpha,\beta)\in\mathbb{R}^7$, we define
\begin{equation}\label{fonctions}
H(A,w)=a Re(w)+b Im (w) +\alpha Re(w^2)+\beta Re(e^{i\eta}w^2)+i[d Re(w)+ e Im (w)]
\end{equation}
and we define
\begin{equation}\label{fonctions2}
F(A,w)=F(w)+\big(0,H(A,w)\big)=\big(w^N,H(A,w)+h(w)\big)
\end{equation}
We let $j,k,l$ be three different integers in $\{1,...,N-1\}$
and we define
$$S(A,w)=\Big(Re\big(H(A,w)-H(A,e^{\frac{2j\pi}{N}i}w)\big)+Re\big(h(w)-h(e^{\frac{2j\pi}{N}i}w)\big),$$
$$Im\big(H(A,w)-H(A,e^{\frac{2j\pi}{N}i}w)\big)+Im\big(h(w)-h(e^{\frac{2j\pi}{N}i}w)\big),$$
$$Re\big(H(A,e^{\frac{2k\pi}{N}i}w)-H(A,e^{\frac{2l\pi}{N}i}w)\big)+Re\big(h(e^{\frac{2k\pi}{N}i}w)-h(e^{\frac{2l\pi}{N}i}w)\big)\Big)$$
We show
\begin{slem}\label{map avec les alpha et beta}
The map $S$ is transverse to $(0,0,0)$; thus, for a generic $A$, $(0,0,0)$ is not attained by $S(A,.)$. 
\end{slem}
\begin{proof}
We set $w=r e^{i\theta}$ and we compute
$$det(\frac{\partial S}{\partial a}, \frac{\partial S}{\partial b},\frac{\partial S}{\partial d})=r^3\Big(\cos\theta-\cos(\theta+\frac{2 j\pi}{N})\Big)\Delta=-2\sin(\frac{j\pi}{N})\sin(\theta+\frac{ j\pi}{N})r^3\Delta
$$
$$det(\frac{\partial S}{\partial a}, \frac{\partial S}{\partial b},\frac{\partial S}{\partial e})=
r^3\Big(\sin\theta-\sin(\theta+\frac{2 j\pi}{N})\Big)\Delta=-2\sin(\frac{j\pi}{N})\cos(\theta+\frac{ j\pi}{N})r^3\Delta
$$
where $$\Delta=
\left|\begin{array}{cc}

 \cos\theta-\cos(\theta+\frac{2j\pi}{N})&\sin\theta-\sin(\theta+\frac{2j\pi}{N})\\
 \cos(\theta+\frac{2k\pi}{N})-\cos(\theta+\frac{2l\pi}{N})&\sin(\theta+\frac{2k\pi}{N})-\sin(\theta+\frac{2l\pi}{N})

 \end{array}\right|$$
 $$=
 \left|\begin{array}{cc}

 2\sin(\theta+\frac{j\pi}{N})\sin(\frac{j\pi}{N})&2\cos(\theta+\frac{j\pi}{N})\sin(\frac{j\pi}{N})\\
 -2\sin(\theta+\frac{(k+l)\pi}{N})\sin(\frac{ (k-l)\pi}{N})&2\cos(\theta+\frac{(k+l)\pi}{N})\sin(\frac{(k-l)\pi}{N})

 \end{array}\right|$$
 \begin{equation}\label{sinus qui peut etre nul}
 =4\sin(\frac{\pi}{N}j)\sin\Big(\frac{\pi}{N}(k-l)\Big)\sin\Big(\frac{\pi}{N}(j-k-l)\Big)
 \end{equation}
 If (\ref{sinus qui peut etre nul}) is zero, then 
 \begin{equation}\label{le premier sinus est nul}
j=k+l
 \end{equation}
We now assume (\ref{le premier sinus est nul}) and we compute 
$$det(\frac{\partial S}{\partial a}, \frac{\partial S}{\partial \alpha},\frac{\partial S}{\partial d})=r^4\Big(\cos\theta-\cos(\theta+\frac{2 j\pi}{N})\Big)\tilde{\Delta}=-2\sin(\frac{j\pi}{N})\sin(\theta+\frac{ j\pi}{N})r^4\tilde{\Delta}
$$
$$det(\frac{\partial S}{\partial a}, \frac{\partial S}{\partial \alpha},\frac{\partial S}{\partial e})=
r^4\Big(\sin\theta-\sin(\theta+\frac{2 j\pi}{N})\Big)\tilde{\Delta}=-2\sin(\frac{j\pi}{N})\cos(\theta+\frac{ j\pi}{N})r^4\tilde{\Delta}
$$
where
$$\tilde{\Delta}=
\left|\begin{array}{cc}

 \cos\theta-\cos(\theta+\frac{2j\pi}{N})&\cos (2\theta)-\cos(2\theta+\frac{4j\pi}{N})\\
 \cos(\theta+\frac{2k\pi}{N})-\cos(\theta+\frac{2l\pi}{N})&\cos(2\theta+\frac{4k\pi}{N})-\cos(2\theta+\frac{4l\pi}{N})

 \end{array}\right|$$ 
$$=
 \left|\begin{array}{cc}

 2\sin(\theta+\frac{j\pi}{N})\sin(\frac{j\pi}{N})&2\sin(2\theta+\frac{2j\pi}{N})\sin(\frac{2j\pi}{N})\\
 -2\sin(\theta+\frac{j\pi}{N})\sin(\frac{ (k-l)\pi}{N})&-2\sin(2\theta+\frac{2j\pi}{N})\sin(\frac{2(k-l)\pi}{N})

 \end{array}\right|$$
 $$=4\sin(\theta+\frac{j\pi}{N})\sin(2\theta+\frac{2j\pi}{N})
 \left|\begin{array}{cc}

 \sin(\frac{j\pi}{N})&\sin(\frac{2j\pi}{N})\\
 -\sin(\frac{ (k-l)\pi}{N})&-\sin(\frac{2(k-l)\pi}{N})

 \end{array}\right|$$ 
\begin{equation}\label{avalanche de sinus}
=-16\sin(\theta+\frac{j\pi}{N})\sin(2\theta+\frac{2j\pi}{N})\sin(\frac{j\pi}{N})\sin(\frac{ (k-l)\pi}{N})\sin(\frac{ l\pi}{N})\sin(\frac{ k\pi}{N})
\end{equation} 
The product (\ref{avalanche de sinus}) is not zero unless 
$$\sin(\theta+\frac{j\pi}{N})\sin(2\theta+\frac{2j\pi}{N})=0$$
in which case we redo the above calculations replacing $\frac{\partial}{\partial\alpha}$ by $\frac{\partial}{\partial\beta}$ and get a non-zero determinant. This concludes the proof of
Sublemma \ref{map avec les alpha et beta}. 
 \end{proof} 
 Given $A$, there is a unique $(\lambda,\mu,\gamma)$ such that for every $w$,
 \begin{equation}\label{lambda et mu viennent de A}
 H(A,w)=\lambda w+\mu\bar{w}+Re(\gamma w^2)
 \end{equation}
 Moreover the map 
 $$A\mapsto (\lambda,\mu,\gamma)$$
 defined by (\ref{lambda et mu viennent de A}) is a surjective submersion. Thus Lemma \ref{pas d'intersection des ensembles finis} follows from Sublemma \ref{map avec les alpha et beta}.
 \end{proof}
\section{The $1$-complex $A$ in $\Pi_2$}
We derive from Lemma \ref{points fixes ds les 3 premieres coordonnes} that the set
\begin{equation}
{\mathcal A}=\{(w_1,w_2)\in\mathbb{D}\times \mathbb{D}\slash w_1\neq w_2\ \ \mbox{and}\ \ 
\pi_3\circ F_{\lambda,\mu}(w_1)=\pi_3\circ F_{\lambda,\mu}(w_2)\}
\end{equation}
is a manifold. If $(w_1,w_2)\in\mathcal A$, then 
$$w_1^N=w_2^N$$
and we let $A$ be the subset of $\mathbb{D}\subset\Pi_2$ consisting of the $w_i^N$'s for $(w_1,w_2)$ in $\mathcal A$. Directly by hand or by standard analytic geometry arguments ($A$ is the projection of an anlytic set and is of dimension $1$, hence it is semi-analytic and so it is stratified, see [\L o]), we derive
\begin{lem}\label{structure de A chapeau}
The set $A$ is a $1$-submanifold of $\mathbb{D}$ with a finite set of singular points which we denote $\Sigma(A)$.
\end{lem}
Moreover we have
\begin{lem}\label{q est un point regulier}
The elements of ${\mathcal D}_{\lambda,\mu}$ (cf. \ref{definition de la projection des points doubles}) are regular points of  $A$. 
\end{lem}
\begin{proof}
If $p\in\mathcal{D}_{\lambda,\mu}$, there exists $w\in\mathbb{D}$ and a number $\nu$ with $\nu^N=1$, $\nu\neq 1$ such that 
$p=w^N$ and 
\begin{equation}\label{encore et encore des estimees}
Re(\lambda w+\mu\bar{w}+h(w))=Re(\lambda \nu w+\mu\bar{\nu}\bar{w}+h(\nu w))
\end{equation}
It follows from Lemma \ref{pas d'intersection des ensembles finis} that in a neighbourhood of $p$, $A$ identifies with
$$A_\nu=\{w^N\slash \pi_3(F(w))=\pi_3(F(\nu w))\}.$$
By the transversality arguments we have been using, we see that, for $\lambda,\mu$ generic, the set of $w$'s which verify 
(\ref{encore et encore des estimees}) 
is a $1$-submanifold, hence $A_\nu$ is one too.
\end{proof}

\section{The loop $\Gamma$ in $\mathbb{D}$}
We now construct a closed loop $\Gamma$ in $\mathbb{D}$. It stays in a small circle around the origin but leaves it to circle around the the points of ${\mathcal D}_{\lambda,\mu}$. We require $\Gamma$ to always meet $A$ transversally. The closed loop $\Gamma$ splits $\mathbb{D}$ into two connected components, $U_0$ and $U_1$ and we require
\begin{itemize}
\item
the origin $0$ and the elements of ${\mathcal D}_{\lambda,\mu}$ are all in $U_0$
\item
the points of $X_{\lambda,\mu}$ (cf. \ref{definition de X lambda mu} for the definition of $X_{\lambda,\mu}$) are all in $U_1$: this is possible since $X_{\lambda,\mu}$ and ${\mathcal D}_{\lambda,\mu}$ do not intersect. 
\end{itemize}
We can now consider three knots in cylinders, namely
\begin{equation}
K=F(\mathbb{D})\cap \pi_2^{-1}(\partial\bar{\mathbb{D}}_2)\ \ \ \ \ K_{\lambda,\mu}=F_{\lambda,\mu}(\mathbb{D})\cap \pi_2^{-1}(\partial\bar{\mathbb{D}}_2)\ \ \ \ \ \hat{K}_{\lambda,\mu}=F_{\lambda,\mu}(\mathbb{D})\cap \pi_2^{-1}(\Gamma)
\end{equation}
We claim that they are all isotopic. For $K$ and $K_{\lambda,\mu}$ to be isotopic, it is enough to take $\lambda$ and $\mu$ small enough. \\
Since there are no double points in $\pi_2^{-1}(U_1)$, the set $M_1=F_{\lambda,\mu}(\mathbb{D})\cap \pi_2^{-1}(U_1)$ is a submanifold of $\mathbb{R}^4$; moreover, if $m\in M_1$, a vector $T$ in $\Pi_2$ has a unique lift in $T_mM_1$. Thus, if we smoothly deform $\Gamma$ to $\partial \bar{\mathbb{D}}$, we can lift this deformation into an isotopy between $K_{\lambda,\mu}$ and $\hat{K}_{\lambda,\mu}$.
\subsection{Construction of $\Gamma$}
It is made of three pieces:
\subsubsection{The circles $\Gamma_i$'s around the  points in ${\mathcal D}_{\lambda,\mu}$}\label{les cercles Gamma}
We let
\begin{equation} 
{\mathcal D}_{\lambda,\mu}=\{p_1,...p_n\}
\end{equation}
The indexing $i$ is chosen so that 
\begin{equation}
arg(p_1)\geq arg(p_1)\geq...\geq arg(p_n)
\end{equation}
For every $i=1,...,n$, the point $p_i$ is a regular point of $A$ (cf. Lemma \ref{q est un point regulier}) so we can pick a small circle 
$\Gamma_i$ in $\mathbb{D}\subset\Pi_2$ centered at $p_i$ and such that 
\begin{enumerate}
\item
the disk bounded by $\Gamma_i$ does not contain any point in $\Sigma(A)$ or a point in ${\mathcal D}_{\lambda,\mu}$ different from $p_i$
\item
$\Gamma_i$ and $A$ meet transversally at two points:
\begin{equation}\label{two points}
\Gamma_i\cap A=\{P_i,Q_i\}
\end{equation}
\end{enumerate}
\subsubsection{The circle $C_\rho$ around the origin}
We pick a small positive number $\rho$; we will indicate below how small we need $\rho$ to be but for the moment we only require 
\begin{equation}
\rho<\frac{1}{2}\inf |p_i|
\end{equation}
and we let $C_\rho$ be the circle in $\mathbb{D}_2$ centered at the origin and of radius $\rho$. 
\subsubsection{The ${\mathcal T}_i$'s between $C_\rho$ and the $\Gamma_i$'s}\label{les mathcal T}
We pick a point $u_i$ on
$\Gamma_i$ different from $P_i,Q_i$.
For every $i$, we pick a path $L_i$  between $u_i$ and $C_\rho$ and a small closed tubular neighbourhood ${\mathcal T}_i$ of $L_i$. \\
We pick the $\mathcal T_i$'s disjoint from one another. Moreover we require for every $i$ that
\begin{enumerate}
\item
$\mathcal T_i$ does not contain any point of ${\mathcal D}_{\lambda,\mu}$ or $\Sigma(A)$
\item
$\mathcal T_i\cap C_\rho\cap A=\varnothing$
\item
$\mathcal T_i\cap \Gamma_i$ does not contain $P_i$ and $Q_i$
\item
the boundary $\partial\mathcal T_i$ meets $A$
transversally.
\end{enumerate}
\subsubsection{Conclusion: the loop $\Gamma$ and the knot/braid $\hat{K}_\lambda$}
To go along the loop $\Gamma$, we start at a point $X_0$ in $C_\rho$ which does not belong to $A$. We follow $C_\rho$ counterclockwise; everytime we meet a ${\partial\mathcal T}_i$, we go along it until we meet $\Gamma_i$; then we follow $\Gamma_i$ till we come to the next component of ${\partial\mathcal T}_i$ which we follow back to $C_\rho$. 
\section{The crossing points of $\hat{K}_{\lambda}$}
We now write $\hat{K}_{\lambda}$ as a braid $\beta$.\\
We denote by $\Pi_{34}$ the plane in $\mathbb{R}^4$ generated by the last two coordinates.
If $\gamma$ is a point of $\Gamma$,  there are $N$ points $\tilde{\gamma}_1,\tilde{\gamma}_2,...,\tilde{\gamma}_N$ in $\Pi_{34}$ such that for all $i$,
$$(\gamma,\tilde{\gamma}_i)\in\hat{K}_{\lambda}$$ 
A crossing point $\gamma^{(0)}$ of  $\hat{K}_{\lambda}$ is a point where the $Re(\tilde{\gamma}^{(0)}_i)$'s  take less than $N$ distinct values, i.e. there are two different points $(\gamma^{(0)},\tilde{\gamma}^{(0)}_i)$ and $(\gamma^{(0)},\tilde{\gamma}^{(0)}_j)$ with the same first three coordinates. In other words, a crossing point occurs when $\Gamma$ meets $A$.\\
To formalize this, 
we parametrize $\Gamma$ as
\begin{equation}
\gamma:[0,2\pi]\longrightarrow \mathbb{D}
\end{equation}
with $\gamma(\theta_0)=\gamma^{(0)}$. We renumber the $i$'s and we reparametrize the $\tilde{\gamma}_i$'s in a neighbourhood of $\theta_0$ so that
$$Re\big(\tilde{\gamma}_1(\theta_0)\big)\geq Re\big(\tilde{\gamma}_2(\theta_0)\big)\geq...\geq Re\big(\tilde{\gamma}_k(\theta_0)\big)=Re\big(\tilde{\gamma}_{k+1}(\theta_0)\big)\geq...\geq Re\big(\tilde{\gamma}_N(\theta_0)\big)$$
This gives us the braid generator
\begin{equation}\label{sigma k}
\sigma_k^{\eta(\gamma_0)}
\end{equation}
with $\eta(\gamma_0)\in\{-1,+1\}$. \\
Note that if there are two different integers $i,j$, with $1\leq i,j\leq N-1$ such that
$$Re\big(\tilde{\gamma}_i(\theta_0)\big)=Re\big(\tilde{\gamma}_{i+1}(\theta_0)\big)\ \ \mbox{and}\ \ \ Re\big(\tilde{\gamma}_j(\theta_0)\big)=Re\big(\tilde{\gamma}_{j+1}(\theta_0)\big)$$
then $|i-j|\geq 2$, which implies that the corresponding $\sigma_i^{\pm}$ and $\sigma_j^{\pm}$ commute; thus it does not matter in which order we write them in the expression of $\beta$.\\
The sign $\eta(\gamma_0)$ of the crossing point in (\ref{sigma k}) is the sign of
\begin{equation}\label{signe d'un point de croisement}
[Im\big(\tilde{\gamma}_{k+1}(\theta_0)\big)-Im\big(\tilde{\gamma}_k(\theta_0)\big)][Re\big(\tilde{\gamma}'_k(\theta_0)\big)-Re\big(\tilde{\gamma}'_{k+1}(\theta_0)\big)]
\end{equation}
This sign is well-defined: the first factor in (\ref{signe d'un point de croisement}) is non zero, otherwise we would have a double point of $F_{\lambda,\mu}$ and we have assumed that none of the double points of $F_{\lambda,\mu}$ project to a point in $\Gamma$.\\
Let us see why the second factor of (\ref{signe d'un point de croisement}) is non-zero. The planes\\ $\pi_3\big(T_{(\gamma(\theta_0),\tilde{\gamma}_k(\theta_0))}F_{\lambda,\mu}(\mathbb{D})\big)$ and $\pi_3\big(T_{(\gamma(\theta_0),\tilde{\gamma}_{k+1}(\theta_0))}F_{\lambda,\mu}(\mathbb{D})\big)$ are transverse (see Lemma \ref{points fixes ds les 3 premieres coordonnes}) so they intersect in a line generated by a vector $X$ which projects to a vector tangent to $A$. The vector $\Big(\gamma'(\theta_0),Re \big(\tilde{\gamma}_k'(\theta_0)\big)\Big)$  - resp. $\Big(\gamma'(\theta_0),Re \big(\tilde{\gamma}_{k+1}'(\theta_0)\big)\Big)$ - completes $X$ in a basis of $\pi_3(T_{(\gamma(\theta_0),\gamma_k(\theta_0))})$ - resp. $\pi_3(T_{(\gamma(\theta_0),\gamma_{k+1}(\theta_0))})$. It follows that 
$$Re \big(\tilde{\gamma}_k'(\theta_0)\big)\neq Re \big(\tilde{\gamma}_{k+1}'(\theta_0)\big)$$
and the sign (\ref{signe d'un point de croisement}) is well-defined.
\\
\\
We now examine the three types
of crossing points.
\subsubsection{On the circle $C_\rho$}
We first investigate the crossing points of the braid
\begin{equation}\label{petites tresses}
w\mapsto (w^N,\lambda w)\ \ \ \ \mbox{(resp.}\ \ \ \ w\mapsto (w^N,\mu \bar{w}))
\end{equation}
Without loss of generality, we assume that $\lambda$ and $\mu$ are real so the crossing points of the braids are given by
\begin{equation}\label{egalite des cosinus}
\cos \frac{2\pi}{N}(\theta+k)=\cos\frac{2\pi}{N}(\theta+l)
\end{equation}
for $k,l\in\{1,...,N-1\}$ and $\theta\in[\zeta,1+\zeta]$, where $\zeta$ is a small positive number which we introduce to avoid crossing points at the endpoints of the interval. We get two values of $\theta$ for (\ref{egalite des cosinus}), namely
$$\theta_1=\frac{1}{2},\ \ \ \theta_2=1.$$
The integers $k,l$ appearing in (\ref{egalite des cosinus}) verify $k+l=N-1$ (resp. $k+l=N-2$) for $\theta_1$ (resp. $\theta_2$). The corresponding values for $\cos \frac{2\pi}{N}(\theta+k)$ are 
$$\mbox{the}\  \cos (2s+1)\frac{\pi}{N}\mbox{'s with }0\leq 2s\leq N-2$$
 $$\mbox{(resp. the }\cos (2s+2)\frac{\pi}{N}\mbox{'s with }0\leq 2s\leq N-3)$$
 Thus the $\cos \frac{2\pi}{N}(\theta+k)$'s go through the values of 
$\cos\frac{\pi}{N}m$ with $1\leq m\leq N-1$.  We conclude: the crossing points above $\theta_1$ (resp. $\theta_2$) correspond to the braid generators  $\sigma^{\pm 1}_{2k+1}$, $1\leq 2k+1\leq N-1$ (resp. $\sigma^{\pm 1}_{2k}$, $1\leq 2k\leq N-1$).\\
\\
It follows from  (\ref{signe d'un point de croisement}) that a crossing point $(\theta_1,\theta_2)$ of $w\mapsto (w^N,\lambda w)$ (resp. $w\mapsto (w^N,\mu \bar{w})$) is of the same sign as
$$(\sin\theta_1-\sin\theta_2)^2\ \ \ \ \mbox{(resp.}\ \ \ \ -(\sin\theta_1-\sin\theta_2)^2)$$
hence they are all positive (resp. all negative).\\
Unlike for the braids (\ref{petites tresses}) the crossing points of $\beta$ on $C_\rho$ will not all occur above the same two points of $C_\rho$; however, 
if $\rho$ is small enough and $\frac{\lambda}{\mu}$ is large enough or small enough, the pieces in $\beta$ corresponding to the crossing points of $C_\rho$ are given by the braids (\ref{petites tresses}): if  that is, the crossing points of $\hat{K}$ on $C_\rho$ translate into the two pieces of $\beta$ described in 1. and 2. of Th. \ref{desingularisation par bands}.
\subsubsection{On $\Gamma_i$}
We recall that $m_i$ is the double point of $F_{\lambda,\mu}$ which projects to the center of $\Gamma_i$. There exist $w_1,w_2\in\mathbb{D}$,
$w_1\neq w_2$ with
$$F_{\lambda,\mu}(w_1)=F_{\lambda,\mu}(w_2)=m_i.$$
We pick a neighbourhood $V_1$ of $w_1$ (resp. $V_2$ of $w_2$) in $\mathbb{D}$. We know that $\pi_3(F_{\lambda,\mu}(V_1))$ and  $\pi_3(F_{\lambda,\mu}(V_2))$
intersect transversally; the curve\\ $\pi_3(F_{\lambda,\mu}(V_1))\cap\pi_3(F_{\lambda,\mu}(V_2))$ projects to $A$ on $\Pi_2$. We also know that  $\Gamma_i$ meets $A$ exactly at two points $P_i,Q_i$ (cf. \S \ref{les cercles Gamma}): the preimages of $P_i$ and $Q_i$ on $\pi_3(F_{\lambda,\mu}(V_1))\cap\pi_3(F_{\lambda,\mu}(V_2))$ give us two braid generators $\sigma_k^{\pm}$ (with the same $k$).\\
We know from Lemma \ref{pas d'intersection des ensembles finis} that these are the only braid generators corresponding to the crossing points $P_i$ and $Q_i$; to get a complete picture of the braid above $\Gamma_i$, we just need to figure out the sign of each of the two $\sigma_k^{\pm}$'s: 
\begin{lem} Let $Q$ be one of the crossing points of
$\beta$ on the circle $\Gamma_i$. Let $m_i\in\mathbb{R}^4$ be the double point of $F_{\lambda,\mu}$ which projects to $p_i$.  
The sign of the crossing point $Q$ is equal to the sign of the double point $m_i$ as a double point of $F_{\lambda,\mu}$.
\end{lem}
\begin{proof}
We let $T_0$ and $T_1$ be the two tangent planes to $F_{\lambda,\mu}(\mathbb{D})$ at $m_i$ and we construct positive bases of $\mathbb{R}^4$, $T_0$ and $T_1$.\\
Since the planes $T_0$ and $T_1$ intersect transversally, the planes $\pi_3(T_0)$ and $\pi_3(T_1)$ also intersect transversally. We let $U$ be a vector in $\Pi_3$ generating $\pi_3(T_0)\cap\pi_3(T_1)$; since $\pi_2\circ F_{\lambda,\mu}(w)=w^N$, $\pi_2\circ F_{\lambda,\mu}$ is a local immersion outside of $0$ and $U$ projects to a non-zero vector $u$ in $\Pi_2$ which is tangent to $A$.\\ We let $v$ be a vector tangent to $\Gamma$ at 
$Q$ oriented in the direction of $\Gamma$; possibly after changing $u$ in $-u$,  $(u,v)$ is a positive basis of $\Pi_2$ (we remind the reader that we have assumed that $A$ and $\Gamma$ meet transversally). 
\\Because $\pi_2\circ F_{\lambda,\mu}$ is a local immersion outside of $0$, there exists a unique $u_i\in P_i$ and a unique $v_i\in P_i$ with 
$$\pi_2(u_i)=u\ \ \ \ \pi_2(v_i)=v$$
Moreover $\pi_2\circ F_\lambda$ preserves the orientation, hence the basis $(u_0,v_0)$ (resp. $(u_1,v_1)$) is a positive basis of $T_0$ (resp. $T_1$). \\
We let $(e_1,e_2,e_3,e_4)$ be an orthonormal positive basis of $\mathbb{R}^4$ with $e_1,e_2$ in $\Pi_2$ and we define
another positive basis of $\mathbb{R}^4$ by
\begin{equation}\label{base de R4}
{\mathcal B}=(u,v,e_3,e_4).
\end{equation}
We write the coordinates in ${\mathcal B}$ of the vectors in the bases base vectors of $T_0$ and $T_1$, namely
$$u_0=(1,0,\alpha,\gamma)\ \ \ \  v_0=(0,1,\beta,\delta)$$
$$u_1=(1,0,\alpha,\gamma')\ \ \ \ v_1=(0,1,\beta',\delta')$$
and we compute the
determinant
\begin{equation}\label{determinant quatre quatre}
det(u_0,v_0,u_1,v_1)=
\left |
\begin{array}{cccc}
1&0&1&0\\
0&1&0&1\\
\alpha&\beta&\alpha&\beta'\\
\gamma&\delta&\gamma'&\delta' \end{array} \right|= -
(\beta-\beta')(\gamma-\gamma')
\end{equation}
\\
We now recover from (\ref{determinant quatre quatre}) the sign of the crossing points of the braid given by   (\ref{signe d'un point de croisement}). \\
For $i=0,1$, we let $V_i$ be a small disk in $F_{\lambda,\mu}(\mathbb{D})$ tangent to $T_i$. We  denote again the two strands which meet at the crossing point by the coordinates in $\mathbb{C}\oplus\mathbb{C}$: $(\gamma(\theta),\tilde{\gamma}_{k}(\theta))$ and $(\gamma(\theta),\tilde{\gamma}_{k+1}(\theta))$. Locally, one is the lift of $\Gamma$ to $V_0$ and the other one is the lift of $\Gamma$ to $V_1$.\\ 
Since the $v_i$'s both project to $v$, we derive that $v_0$ (resp. $v_1$) is the vector tangent to $F_{\lambda,\mu}\cap\pi_2^{-1}(\Gamma)$ above 
$Q$ on $V_0$ (resp. $V_1$). Hence
$$Re\big(\tilde{\gamma}_{k}'(\theta_0)\big)-Re\big(\tilde{\gamma}_{k+1}'(\theta_0)\big)$$
 has the same sign as $\beta-\beta'$.\\
\\
We now use the fact that $Q$ belongs to $A$. Since $p_i$ is a regular point of $A$, $A$ is parametrized near $q_i$ by
\begin{equation}\label{parametrisation de A}
a:t\mapsto q_i+tu+o(t^2)
\end{equation}
Since $(u,v)$ is a positive basis of $\Pi_2$ and $v$ is tangent to $\Gamma$ at $Q$, the point $Q$ is on the side of the positive $t$'s in (\ref{parametrisation de A}). The lift of $A$ to $V_0$ (resp. $V_1$) is parametrized by
\begin{equation}\label{parametrisation des eta}
\tilde{a}_0(t)=m_i+tu_0+o(t^2)\ \ \ \ \ \tilde{a}_1(t)=m_i+tu_1+o(t^2)
\end{equation}
Thus, if we have taken
$\Gamma_i$ small enough, $Im(\tilde{\gamma}_k(\theta_0))-Im(\tilde{\gamma}_{k+1}(\theta_0))$ is of the same sign as $(u_0)_4-(u_1)_4=\gamma-\gamma'$.
\end{proof}
Thus the circle $\Gamma_i$ contributes $\sigma_k^{2\epsilon(q_i)}$ to the braid. 
\subsubsection{On $\partial\mathcal T_i$}
We proceed as in [Ru 1]. \\
If ${\mathcal T}_i$ is a small enough neighbourhood, the map $F_{\lambda,\mu}:F_{\lambda,\mu}^{-1}({\mathcal T}_i)\longrightarrow {\mathcal T}_i$ is a covering, hence $F_{\lambda,\mu}^{-1}({\mathcal T}_i)$ is a disjoint union of $N$ copies of $L_i\times[-\eta,+\eta]$ for a small $\eta>0$.\\
If $q_0$ is a point in $L_i\cap A$, there are two points $q_1$ and $q_2$ close to $q_0$ in $\mathcal T_i\cap A$, one in each component of $\mathcal T_i$. If the $k$-th and $(k+1)$-th leaf of $\pi_3\circ F_{\lambda,\mu}(\mathbb{D})$ coincide above $q_0$, the same is true for $q_1$ and $q_2$. Hence $q_1$ and $q_2$ each give us a braid generator $\sigma_k^{\pm}$ for $\beta$.\\
These two $\sigma_k^{\pm}$'s have opposite signs. Indeed, if we look at formula (\ref{signe d'un point de croisement}), the factors $Im\big(\tilde{\gamma}_{k+1}(\theta_0)\big)-Im\big(\tilde{\gamma}_k(\theta_0)\big)$ take the same sign for both $q_1$ and $q_2$, whereas the factors
$Re\big(\tilde{\gamma}'_k(\theta_0)\big)-Re\big(\tilde{\gamma}'_{k+1}(\theta_0)\big)$ take opposite signs.\\
\\
Putting all the $\sigma_k^{\pm 1}$'s together, we get an element $b_i\in B_N$ such that the piece of the braid which consists in going along ${\mathcal T}_i$, around $\Gamma_i$ and back along ${\mathcal T}_i$ can be written as
\begin{equation}
b_i \sigma_{k(i)}^{2 \epsilon(Q)} b_i^{-1}
\end{equation}
where $k(i)$ is an integer in $\{1,...,N-1\}$ and $\epsilon(Q)$ is the sign of the crossing point. \\
We get the terms in the braid of Th. \ref{desingularisation par bands} 3 and the proof of Th. \ref{desingularisation par bands} is completed.

\footnotesize{Marina.Ville@lmpt.univ-tours.fr\\LMPT,
 Universit\'e de Tours
 UFR Sciences et Techniques
 Parc de Grandmont
 37200 Tours, FRANCE}


\begin{thebibliography}{10}
\bibitem[G-O-R]{gor} R. D. Gulliver II, R. Osserman, H. L. Royden {\it A Theory of Branched Immersions of Surfaces},
Amer. Jour. of Maths,
Vol. 95, No. 4 (1973), pp. 750-812
\bibitem[\L o]{lo} S. \L ojasiewicz, {\it Sur la g\'eom\'etrie semi- et sous-analytique}, Ann. de l'Institut Fourier, 43(5) 1993, 1575-1595.
\bibitem[Mi]{mi} J. Milnor, {\it Singular points of complex hypersurfaces}, Ann. of mathematics studies, PUP (1969).
\bibitem[Ru 1]{ru1}    L. Rudolph, {\it Algebraic functions and closed braids}, Topology 22(2) (1983) 191-202 and arxiv.org/pdf/math/0411316
\bibitem[Ru 2]{ru2}  L. Rudolph, {\it Braided surfaces and Seifert ribbons for closed braids}, Comment. Math. Helvetici 58 (1983) 001-037
\bibitem[S-V]{sv1} M. Soret, M.Ville, {\it Singularity Knots of Minimal Surfaces in $R^4$}, Jour. of Knot theory and its ramifications, 20 (4), (2011), 513-546. 
\bibitem[Vi 1]{vi1} M. Ville, {\it Branched immersions and braids}, Geom. Dedicata, 140(1), 2009, 145-162.
\bibitem[Vi 2]{vi2} M. Ville, {\it Desingularization of branch points of minimal disks in $\mathbb{R}^4$}, http://arxiv.org/pdf/1412.0589

\end{thebibliography}
\end{document}